\documentclass[preprint,12pt]{elsarticle}
\makeatletter
\def\ps@pprintTitle{%
 \let\@oddhead\@empty
 \let\@evenhead\@empty
 \def\@oddfoot{\centerline{\thepage}}%
 \let\@evenfoot\@oddfoot}
\makeatother

\usepackage{graphicx}

\usepackage{amssymb}
\usepackage{amsthm}
\usepackage[fleqn]{amsmath}
\usepackage{mathtools}
\usepackage{cases}
\usepackage{tikz}
\usetikzlibrary{hobby}

\newtheorem{thm}{Theorem}

\newtheorem{lem}[thm]{Lemma}
\theoremstyle{example}
\newtheorem{exm}{Example}
\newtheorem{prop}[thm]{Proposition}
\theoremstyle{definition}

\theoremstyle{remark}
\newtheorem{rem}{Remark}

\journal{}

\begin{document}

\begin{frontmatter}


\title{On the boundary terms in Hardy's inequalities
for $W^{1,p}$ functions}
\author{Ahmed A. Abdelhakim}
\address{Mathematics Department, Faculty of Science, Assiut University, Assiut 71516 - Egypt\\
Email: ahmed.abdelhakim@aun.edu.eg}
\begin{abstract}
\indent With the help of a radially invariant
vector field, we derive inequalities of
the Hardy kind, with no boundary terms,
for $W^{1,p}$ functions on bounded star
domains. Our results are not obtainable
from the classical inequalities for
$W^{1,p}_{0}$ functions. Unlike in $W^{1,p}_{0}$,
our inequalities admit
maximizers that we describe explicitly.
\end{abstract}
\end{frontmatter}
\section{Introduction}
Let $1\leq p<n$ and let $u \in C^{\infty}_{0}
\left(\Omega\right)$ where
$\Omega$ is a $C^{1}$ bounded domain in $\mathbb{R}^n$. Then
\begin{equation}\label{ineq0}
\left\| \frac{u}{|x|} \right\|_{L^{p}\left(\Omega\right)}
\leq c_{p,n}\| \nabla u \|_{L^{p}\left(\Omega\right)}
\end{equation}
where $ c_{p,n}={p}/\left({n-p}\right).$
The inequality (\ref{ineq0}) is a version
of the wellknown Hardy's inequality.
It is interesting, obviously,  if the origin belongs to $\Omega$.
The constant $c_{p,n}$ in (\ref{ineq0}) is found to be optimal \cite{Balinsky},
yet not attained in the corresponding
Sobolev space $W_{0}^{1,p}\left(\Omega\right)$.
This is a motivation to look for a remainder term.
A substantial improvement
of (\ref{ineq0}) when $p=2$ was obtained by Brezis and Vazquez \cite{Brezis} who proved the inequality
\begin{equation}\label{ineq2}
\left(\int_{\Omega} |\nabla u|^2 dx-
\left(\frac{n-2}{2}\right)^2 \int_{\Omega}
 \frac{u^2}{|x|^2} dx \right)^{\frac{1}{2}}\geq
 \frac{C(q,n)}{|\Omega|^{\frac{1}{q}-\frac{n-2}{2}}}
\left(\int_{\Omega} |u|^q dx \right)^\frac{1}{q},
\end{equation}
$1 < q < n/(n-2)$, with the constant $C(q,n)$ optimal when $\Omega$ is a ball centered at the origin and $q=2$,
but again, never achieved in $H^{1}_{0}\left(\Omega\right)$.
Similar improvements for Hardy's inequalities where a nonnegative correction term is introduced
followed this result (\cite{Barbatis}, \cite{Brezis1,Brezis2},\cite{Gazzola}, \cite{Vazquez2}). But these mainly targeted versions of (\ref{ineq0}) that involve the distance from the boundary as opposed to the distance from the origin or treated the corresponding
$L^p$ cases. Filippas and Tertikas
\cite{Filippas1}  optimized   (\ref{ineq2}), in a certain sense, in terms of correction terms and showed that the best constants in their improvements can not be achieved in $H^{1}_{0}(\Omega).$
Later, N. Ghoussoub, A. Moradifam \cite{Ghoussoub1}
characterized radially
symmetric potentials $V$ and best constants $c(V)$
for the Hardy inequality
\begin{equation*}
\int_{\Omega} |\nabla u|^2 dx-
\left(\frac{n-2}{2}\right)^2 \int_{\Omega}
 \frac{u^2}{|x|^2} dx \geq
 {c(V)}\int_{\Omega} V(|x|) |u|^2 dx,
 \qquad u\in H^{1}_{0}(\Omega).
\end{equation*}
These results were furthered in (\cite{Alvino},\cite{cowan,Cuomo,Devyvera}, \cite{Ghoussoub2}, to name a few).\\
\indent
The unattainability of the aforementioned optimal constants persists regardless of the behaviour  of
$W^{1,p}_{0}$ functions near the boundary.
Notice that (\ref{ineq0}) does not hold
for functions constant on $\Omega$, and therefore not true for all $u\in W^{1,p}\left(\Omega\right).$
Any inequality of the Hardy type for
$u\in W^{1,p}\left(\Omega\right)$ will certainly involve a boundary term. \\
\indent The discussion above invokes a question: Does the inequality (\ref{ineq0})
hold true on a space larger than (or different from) $W_{0}^{1,p}\left(\Omega\right)$ on which the constant $c_{p,n}$, or a bigger one, is optimal and achieved ?\\
\indent Recently S. Machihara et al. \cite{citeballs0}
gave variants of (\ref{ineq0}) on the ball
$B(R)=\left\{x\in \mathbb{R}^n: |x|<R\right\}$, valid
for $H^{1}$ functions, namely
\begin{align}
&\hspace{-0.7 cm}\label{ineq1} \int_{B(R)}\frac{\left|u(x)-u\left (R\frac{x}{|x|}
 \right)\right|^{2}}{|x|^2}
 dx
\leq \left(\frac{2}{n-2}\right)^{2}\int_{B(R)} \left|\frac{x}{|x|}\cdot\nabla u \right|^{2} dx,
\quad n\geq 3,\\
&\hspace{-0.7 cm}
\label{critical1}
\int_{B(R)}\frac{\left|u(x)-u\left (R\frac{x}{|x|} \right)\right|^{2}
}{|x|^2\left| \log{\frac{R}{|x|}} \right|^2}
 dx
\leq 4\int_{B(R)} \left|\frac{x}{|x|}\cdot\nabla u \right|^{2} dx,\quad n=2.
\end{align}
The novelty in inequalities (\ref{ineq1}) and (\ref{critical1}) lies in obtaining inequalities of Hardy type on balls for a function $u\in H^{1}\left(\mathbb{R}^n\right)$ with no boundary terms.
Observe the identity
\begin{equation*}
x\cdot\nabla u\left (R\frac{x}{|x|} \right)=0.
\end{equation*}
This idea is celebrated in
\cite{citeballs2,citeballs1},
\cite{citeballs3,citeballs4,citeballs5,citeballs6}.\\\\
\indent First we show that (\ref{ineq1})
can not be obtained from (\ref{ineq0}) directly,
and neither can (\ref{critical1}) be deduced from
the inequality (\cite{Ioku2016}):
\begin{equation}\label{ineqc}
\int_{\Omega}\frac{|u(x)|^{n}}{|x|^{n}
\left(\log{\frac{M}{|x|}}\right)^{n}}\,dx\leq
\left(\frac{n}{n-1}\right)^{n}
\int_{\Omega} \left|\frac{x}{|x|}\cdot\nabla u\right|^{n} dx,\quad M=\sup_{\Omega}{|x|},
\end{equation}
that holds for all $u\in W^{1,n}_{0}(\Omega)$,
$n\geq 2$. Precisely, if  $v(x):=u(x)-u\left(x/|x|\right)$
then $v$ is not necessarily in $H^{1}\left(B(1)\right)$ whenever
$u\in H^{1}\left(\mathbb{R}^n\right).$ In fact $u\left(x/|x|\right)$ is not always in $L^{2}\left(B(1)\right)$ when $u\in L^{2}\left(\mathbb{R}^n\right)$. See Proposition \ref{cntbtnd}.\\
\indent
Second, with the help of a vector field $f$ satisfying a particular boundary value problem, we extend inequalities (\ref{ineq1}) and
(\ref{critical1}) to bounded $C^{1}$ star-shaped with respect to the origin domains.
This is easy for inequalities like (\ref{ineq0})
obtainable from their analogues on balls:\\
\indent Since $c_{p,n}$ is independent of $\Omega$, then a standard proof \cite{Balinsky} implies
\begin{equation}\label{br}
\int_{B(R)} \frac{|u|^{p}}{|x|^{p}}dx
\leq c^{p}_{p,n}\int_{B(R)} |\nabla u|^{p} dx,\qquad
u\in H^{1}_{0}(B(R)).
\end{equation}
Set $R=\max_{x\in \Omega}{\left|x\right|}$ and simply extend $u\in H^{1}_{0}\left(\Omega\right)$ to
\begin{equation*}
\bar{u}(x):=\left\{
              \begin{array}{ll}
                $u(x)$, & \hbox{$x \in \Omega$;} \\
                $0$, & \hbox{\text{otherwise}.}
              \end{array}
            \right.
\end{equation*}
Then $\bar{u}\in H^{1}_{0}\left(B\left(\max_{x\in \Omega}{\left|x\right|}\right)\right)$
and (\ref{ineq0}) follows from (\ref{br}).
This extension by zero argument (\cite{Ioku2016}, Remark 1.3) does not work for inequalities (\ref{ineq1}) and
(\ref{critical1}). The term $u(R x/|x|)$
that grants $v$ zero ``trace" on
$\partial B(R)$ becomes idle
and the field $x/|x|$ needs to be replaced by
another radially invariant field that will definitely
depend on the domain. As a result, spherical coordinates
used in \cite{citeballs0} can be no
longer helpful. As with
(\ref{ineq1}) and (\ref{critical1}),
by Proposition \ref{cntbtnd}, our inequalities
are unobtainable from
(\ref{ineq0}) when $1\leq p<n$, $n\geq3$ or from
(\ref{ineqc}) in the critical case $p=n\geq 2$.
\\ \indent Finally, sharpness
and equality are discussed
using a method different from Ioku et al.'s in
\cite{citeballs2}, \cite{citeballs1}.
The case $p=1$ missing
in \cite{citeballs1} is also recovered.
We actually identify maximizers for our inequalities
in $W^{1,p}$. We show how exactly the existence of these maximizers depends on both $f$ and the domain.
\section{Main results}
\begin{prop}\label{cntbtnd}
Let $n\geq 2$, $1\leq p<n$, and
assume $\,u\in L^{p}\left(\mathbb{R}^n\right)$.
Then $\,v(x):=u(x)-u\left(x/|x|\right)$
needs not belong to $L^{p}\left(B(1)\right)$.
Moreover, $u\in W^{1,p}\left(\mathbb{R}^n\right)$,
does not guarantee $v\in W^{1,p}\left(B(1)\right)$.
\end{prop}
\begin{proof}
Fix $x_0 \in {\mathbb{R}}^n$ with $ |x_0|=1$.  Let $ 0 \leq \gamma \leq 1$ denote a cutoff  function with $ \gamma(x)=1$ for $\frac{3}{4}<|x|<\frac{5}{4}$ and $\gamma=0$ for $|x|<\frac{1}{2}$ and $ |x|>\frac{3}{2}$. Let
\begin{equation*}
u(x):= \frac{\gamma(x)}{|x-x_0|^{\alpha}},\qquad \alpha>0.
\end{equation*}
We have
\begin{equation*}
\int_{\mathbb{R}^{n}}
|u|^p\,dx\lesssim \int_{\frac{1}{2}\leq|x|\leq\frac{3}{2}}
\frac{dx}{|x-x_0|^{p\alpha}}\lesssim \int_{|x|\leq1}
\frac{dx}{|x|^{p\alpha}}.
\end{equation*}
Thus $u\in L^{p}(B(1))$ if $\alpha<\frac{n}{p}$. Whereas, using spherical coordinates,
\begin{equation*}
\int_{B(1)}
\left|u\left(\frac{x}{|x|}\right)\right|^p\,dx
= \int_{\frac{1}{2}\leq|x|\leq 1}
\frac{dx}{|\frac{x}{|x|}-x_0|^{p\alpha}}
\approx \int_{|\omega|=1}
\frac{d\omega}{|\omega-x_{0}|^{p\alpha}}
\end{equation*}
which converges iff $\alpha<\frac{n-1}{p}$. This excludes $v$ from $L^{p}\left(B(1)\right)$
for $\alpha\geq\frac{n-1}{p}$. Also
\begin{eqnarray*}
\int_{\mathbb{R}^{n}} |\nabla u|^p\,dx
&\lesssim&
\int_{
\frac{1}{4}\leq ||x|-1|\leq\frac{1}{2}}
\frac{|\nabla  \gamma|^{p}}{|x-x_0|^{p\alpha}}\,dx+
\int_{\frac{1}{2}\leq|x|\leq\frac{3}{2}}
\frac{|\gamma|^{p}}{|x-x_0|^{p(\alpha+1)}}\,dx\\
&\lesssim& 1+\int_{|x|\leq 1}
\frac{dx}{|x|^{p(\alpha+1)}},
\end{eqnarray*}
whence $u\in W^{1,p}\left(\mathbb{R}^n\right)$
whenever $\alpha<\frac{n-p}{p}$. Meanwhile
\begin{eqnarray*}
\nabla\left(u\left(\frac{x}{|x|}\right)\right)
&=&\frac{1}{|x|}
\left(\nabla u\right)\left(\frac{x}{|x|}\right)-
\frac{x\cdot\left(\nabla u\right)\left(\frac{x}{|x|}\right)}{|x|^3}\,x\\
&=&\frac{1}{|x|}\left(
\left(\nabla u\right)\left(\frac{x}{|x|}\right)-
\frac{x}{|x|}\cdot\left(\nabla u\right)\left(\frac{x}{|x|}\right)
\frac{x}{|x|}\right)\\
&=&\frac{1}{r}\left[
\left(\nabla u\right)\left(\omega\right)
-\left(\hat{r}\cdot\left(\nabla u\right)\left(\omega\right)\right)
\hat{r}\right]
\,=\,\frac{1}{r}\nabla_{\omega}u\left(\omega\right)
\end{eqnarray*}
in spherical coordinates. Consequently
\begin{eqnarray*}
\int_{B(1)}\left|\nabla
u\left(\frac{x}{|x|}\right)\right|^{p}
dx&\approx&
\int_{|\omega|=1}
\left|\nabla_{\omega}u\left(\omega\right)\right|^{p}
d\omega\,=\,
\int_{|\omega|=1}\frac{d\omega}{|\omega-x_0|^{p(\alpha+1)}}.
\end{eqnarray*}
Therefore $v\notin W^{1,p}(B(1))$ for any
$\alpha\geq {(n-p-1)}/{p}$.
\end{proof}
The upcoming lemma provides
an implemental estimate of weak derivatives.
\begin{lem}\label{recover1}(\cite{Balinsky})
Let $1\leq p \leq \infty$ and $g\in W^{1,p}\left(\Omega\rightarrow \mathbb{R}\right)$.
Then
$\,\displaystyle |x\cdot \nabla |g||\leq |x\cdot \nabla g|.$
\end{lem}
\begin{thm}\label{thm1}
Let $n\geq 3$, $1\leq p<n$, and let $u\in W^{1,p}(\Omega)$
where $\Omega$ is a bounded $C^{1}$ domain
star-shaped with respect to the origin in $\mathbb{R}^{n}$. Suppose $f\in C^{1}\left(\bar{\Omega}_0 \rightarrow \mathbb{R}^n \right)\cap
L^{\infty}\left(\Omega\rightarrow \mathbb{R}^n \right)$ solves the boundary value problem
\begin{equation}
\begin{split}\label{fbvp}
\left\{
  \begin{array}{ll}
\vspace{0.2 cm} \left(x\cdot \nabla \right) f(x)=0, & \hbox{$x \in\Omega_{0}$;} \\
    f(x)=x, & \hbox{$x \in\partial\Omega$}
  \end{array}
\right.
\end{split}
\end{equation}
where $\Omega_{0}$ denotes $\Omega\setminus\{0\}$.
Then
\begin{eqnarray}\label{mainineq1}
\int_{\Omega}
\frac{|u-u\circ f|^{p}}{|x|^{p}}\,dx
&\leq& \left(\frac{p}{n-p}\right)^{p} \int_{\Omega}\left|\frac{x}{|x|} \cdot \nabla u \right|^{p} \,dx.
\end{eqnarray}
\end{thm}
\begin{proof}
The proof is standard. By density \cite{adams}, we may argue assuming $u\in C^{1}\left(\bar{\Omega}\right)$.
Let $F(x):={|u(x)-u(f(x))|^{p}}/{|x|^{p}}$.
Then, for
$x\neq 0$, we have
\begin{eqnarray*}
\nonumber
F(x)&=&
-|u(x)-u(f(x))|^{p}\;\nabla \left(\frac{1}{|x|}\right)
\cdot
\frac{x}{|x|^{p-1}}\\
&=&-\nabla \cdot\left(F(x)\,x\right)+
\left(n-\left(p-1\right)\right)
F(x)+pF^{\frac{p-1}{p}}(x)
\frac{x}{|x|}\cdot \nabla |u(x)-u(f(x))|
\end{eqnarray*}
which we rewrite as
\begin{equation}\label{qq0}
F(x)=\frac{1}{n-p}\nabla \cdot\left(F(x)\,x\right)
-\frac{p}{n-p}F^{\frac{p-1}{p}}(x)
\frac{x}{|x|}\cdot \nabla |u(x)-u(f(x))|.
\end{equation}
Now, choose $\epsilon>0 $ such that $\Omega\supset B(\epsilon)$. Then
\begin{equation}\label{eps1}
I(\epsilon):=\int_{B(\epsilon)}
F(x)\,dx
\lesssim
\|u-u\circ f\|_{L^{\infty}
\left(\Omega\right)}^{p}\,\epsilon^{n-p}.
\end{equation}
The estimate (\ref{eps1}) helps isolate
the singularity so that, if $\nu$ is the outward pointing normal on $\Omega$, then the divergence theorem yields
\begin{equation}\label{qq1}
J(\epsilon):=\int_{\Omega\setminus B(\epsilon)}\,\nabla \cdot\left(F(x)\,x\right)dx=-\int_{|x|=\epsilon} \frac{|u(x)-u(f(x))|^{p}}{|x|^{p-1}}dS(x),
\end{equation}
as assumption (\ref{fbvp}) ensures
\begin{equation*}
\int_{\partial\Omega} F(x)\,x\cdot \nu(x)dS(x)=0.
\end{equation*}
And from (\ref{qq1}) follows the estimate
\begin{equation}\label{eps2}
\left|J(\epsilon)\right|
\lesssim \|u-u\circ f\|_{L^{\infty}
\left(\Omega\right)}^{p}
\epsilon^{n-p}.
\end{equation}
In addition, applying H\"{o}lder's inequality
\begin{align}
\nonumber
&\hspace{-1 cm}|K(\epsilon)|:=
\left|\int_{\Omega\setminus B(\epsilon)}F^{\frac{p-1}{p}}(x)
\frac{x}{|x|}\cdot
\nabla |u(x)-u(f(x))|\,dx\right|\\
\label{qq2} &\hspace{-0.45 cm}\leq
\left(\int_{\Omega\setminus B(\epsilon)}F(x)
\,dx\right)^{1-\frac{1}{p}}
\left(\int_{\Omega\setminus B(\epsilon)}
 \left|\frac{x}{|x|}\cdot\nabla
 \left|u(x)-u(f(x))\right|\right|^{p}
 \,dx\right)^{\frac{1}{p}}.
\end{align}
But Lemma 2 affirms the pointwise estimate
\begin{equation}\label{rt1}
\left|x \cdot\nabla \left|u(x)-u(f(x))\right|\right|
\leq \left|x \cdot\nabla \left(u(x)-u(f(x))\right)\right|,\qquad x\in \Omega_{0}.
\end{equation}
Also, from (\ref{fbvp}) follows
\begin{equation}\label{rt2}
x \cdot\nabla u(f(x))=
\left(\nabla u\right)(f(x))\cdot
\left(x\cdot\nabla\right)f=0,\qquad x\in \Omega_{0}.
\end{equation}
Returning with (\ref{rt1}) and (\ref{rt2})
to (\ref{qq2}) implies
\begin{equation}\label{qq3}
|K(\epsilon)|\leq
\left(\int_{\Omega\setminus B(\epsilon)}F(x)\,dx\right)^{1-\frac{1}{p}}
\left(\int_{\Omega\setminus B(\epsilon)} \left|\frac{x}{|x|}\cdot\nabla u\right|^{p}\,dx\right)^{\frac{1}{p}}.
\end{equation}
Integrating (\ref{qq0}) over
$\Omega$ we get
\begin{equation}\label{finp}
\int_{\Omega}F(x)\,dx
=I(\epsilon)+
\frac{1}{n-p}J(\epsilon)-\frac{p}{n-p}K(\epsilon).
\end{equation}
Finally, by (\ref{eps1}), $\lim_{\epsilon\rightarrow 0^{+}}I(\epsilon)=0$. Similarly
$\lim_{\epsilon\rightarrow 0^{+}}J(\epsilon)=0$
from (\ref{eps2}). And since $|\nabla u|^{p}\in L^{1}(\Omega)$, then, using the dominated convergence theorem in (\ref{qq3}), we deduce (\ref{mainineq1})
from (\ref{finp}).
\end{proof}
\begin{rem}
Star-shapedness of the domain $\Omega$ with respect to the origin is necessary
for a nontrivial vector field $f$ to satisfy \ref{fbvp}.
\end{rem}
\begin{thm}\label{thm2}
Consider $\,\Omega$ and $f$ of Theorem \ref{thm1}.
Let $M:=\sup_{\Omega}{|x|}$ and assume $u\in W^{1,n}(\Omega)$
with $n\geq 2$. Then
\begin{eqnarray}\label{mainineq2}
\int_{\Omega}\frac{|u-u\circ f|^{n}}{|x|^{n}
\left(\log{\frac{M}{|x|}}\right)^{n}}\,dx\leq
\left(\frac{n}{n-1}\right)^{n}
\int_{\Omega} \left|\frac{x}{|x|}\cdot\nabla u\right|^{n} dx.
\end{eqnarray}
\end{thm}
\begin{proof}
Invoking the density argument \cite{adams}, we may take
$u\in C^{1}\left(\bar{\Omega}\right)$. Since, for
$x\in\Omega_{0}$,
\begin{eqnarray*}
{(n-1)}\frac{|u(x)-u(f(x))|^{n}}{|x|^{n}
\left(\log{\frac{M}{|x|}}\right)^{n}}&=&
|u(x)-u(f(x))|^{n}
\nabla\cdot \left(
\frac{x}{|x|^{n}
\left(\log{\frac{M}{|x|}}\right)^{n-1}}
\right)\\
&=&\nabla\cdot \left(
\frac{|u(x)-u(f(x))|^{n}}{|x|^{n}
\left(\log{\frac{M}{|x|}}\right)^{n-1}}
\,x\right)+\\
&&-n\frac{|u(x)-u(f(x))|^{n-1}}{|x|^{n-1}
\left(\log{\frac{M}{|x|}}\right)^{n-1}}
\frac{x}{|x|}\cdot\nabla|u(x)-u(f(x))|.
\end{eqnarray*}
The rest resembles the proof of Theorem \ref{thm1}
once the components (\ref{cr1})-(\ref{cr3}) below are recognized:\\
\begin{equation}\label{cr1}
\int_{B(\epsilon)}\frac{|u(x)-u(f(x))|^{n}}{|x|^{n}
\left(\log{\frac{M}{|x|}}\right)^{n}}dx
\lesssim \frac{\|u-u\circ f\|_{L^{\infty}
\left(\Omega\right)}^{n}}{
\left(\log{\frac{M}{\epsilon}}\right)^{n-1}}
\rightarrow 0,\quad\text{as}\;\;\epsilon\rightarrow 0^{+}.
\end{equation}
\begin{equation}\label{cr2}
\int_{\partial\Omega}
\frac{|u(x)-u(f(x))|^{n}}{|x|^{n}
\left(\log{\frac{M}{|x|}}\right)^{n-1}}
\left(x\cdot \nu\right)dx=0,
\end{equation}
\begin{align}
\nonumber&-\int_{\Omega\setminus B(\epsilon)}
\nabla\cdot\left(
\frac{|u(x)-u(f(x))|^{n}}{|x|^{n}
\left(\log{\frac{M}{|x|}}\right)^{n-1}}
\,x\right)dx\\
&\label{cr3} =\int_{|x|=\epsilon}
\frac{|u(x)-u(f(x))|^{n}}{|x|^{n-1}
\left(\log{\frac{M}{|x|}}\right)^{n-1}}dS(x)
\lesssim \frac{\|u-u\circ f\|_{L^{\infty}
\left(\Omega\right)}^{n}}{
\left(\log{\frac{M}{\epsilon}}\right)^{n-1}}
\rightarrow 0,\quad\text{as}\;\;\epsilon\rightarrow 0^{+},
\end{align}
an implication of the divergence theorem.
\end{proof}
\section{Applications}\label{apps}
\subsection*{1- On ellipsoids}
Let $a:=(a_{i})\in \mathbb{R}_{+}^{n}$ and let
$|.|_{a}$ denote the norm $|x|_{a}:=\left(\sum_{i=1}^{n}
{a_{i}^{2} x_{i}^{2}}
\right)^{\frac{1}{2}}$, $x=(x_{i})\in \mathbb{R}^{n}$.
Consider the open ellipsoid
$E_{a}:=\left\{x:|x|^2_{a}< 1\right\}$.
The essentially bounded field $f_{a}(x):=x/|x|_{a}$
is smooth on $E_{a}\setminus\{0\}$ and satisfies
(\ref{fbvp}) with $\Omega=E_{a}$. Therefore,
if $u\in W^{1,p}(E_{a})$, $1\leq p<n$, $n\geq 3$, then
\begin{equation*}
\int_{E_{a}}
\frac{|u-u\circ f_{a}|^{p}}{|x|^{p}}\,dx
\leq \left(\frac{p}{n-p}\right)^{p}
\int_{E_{a}}\left|\frac{x}{|x|}
\cdot \nabla u \right|^{p}
\,dx
\end{equation*}
by Theorem \ref{thm1}. Moreover,
applying Theorem \ref{thm2},
if $u \in W^{1,n}(E_{a})$ with
$n\geq 2$ then
\begin{equation*}
\int_{E_{a}}\frac{|u-u\circ f_{a}|^{n}}
{|x|^{n}
\left(\log{\frac{M_{a}}{|x|}}\right)^{n}}\,dx\leq
\left(\frac{n}{n-1}\right)^{n}
\int_{E_{a}} \left|\frac{x}{|x|}
\cdot\nabla u\right|^{n} dx
\end{equation*}
where $ M_{a}:=\sup_{x\in E_{a}}{|x|}=
1/\min_{i}{\,a_{i}}$.
\subsection*{2- On star-shaped domains}
\indent Let the $C^{1}$ domain $\mathcal{C}$
be such that $0\in \mathcal{C}$ and $\mathcal{C}$
is star-shaped w.r.t. $0$.
Suppose $\partial\mathcal{C}$
has the representation
 $r = r(\omega)$ in spherical coordinates.
Then $\mathcal{C}$ and
\begin{equation*}
f(x):=r(\frac{x}{|x|})\frac{x}{|x|}
\end{equation*}
fulfill all prerequisites of both Theorems \ref{thm1}
and \ref{thm2}.
\section{Sharpness and maximizers}
Back to Theorems \ref{thm1} and \ref{thm2}.
If $\,u\in W^{1,p}(\Omega)\cap C(\bar{\Omega})\,$
and $f$ is such that $u\circ f \in W^{1,p}(\Omega)$, then $\,u-u\circ f\in W_{0}^{1,p}(\Omega)$. Hence,
by (\ref{rt2}), the sharpness of the inequalities (\ref{ineq0}) and
(\ref{ineqc}) then assure
the optimality of the constants in (\ref{mainineq1})
and (\ref{mainineq2}), respectively.
Equality in (\ref{mainineq1})
is clear for constant $u$, for which both sides vanish.
The same applies to (\ref{mainineq2}).
In fact if $u$ is radially invariant in $\Omega$
then $u(x)=u(f(x))$ and $x.\nabla u=0$, in which case
both sides of (\ref{mainineq1}) and (\ref{mainineq2})
are zero.\\
\indent  Let us investigate all
candidates for radially varying maximizers
of (\ref{mainineq1}). As intuited from Proposition \ref{cntbtnd},
the existence of a maximizer
in $W^{1,p}(\Omega)$ depends on $f$ and
the geometry of $\Omega$. Now, equality in
(\ref{mainineq1}) requires equality in
(\ref{qq2}) which necessitates the existence of
a $\lambda>0$ such that
\begin{equation}\label{Feq1}
F^{\frac{1}{p}}(x)=\lambda \left|\frac{x}{|x|}\cdot\nabla \left|u(x)-u(f(x))\right|\right|
=\lambda\left|\frac{x}{|x|}\cdot\nabla
u(x)\right|,\quad p>1,
\end{equation}
for almost every $x\in \Gamma$ with
$\Gamma:=\left\{x\in
\Omega_{0}:\, u(x)-u(f(x))\neq 0\right\}$.
Comparing (\ref{Feq1})
against (\ref{mainineq1}) with an equality implies
\begin{equation*}
 \lambda=\frac{p}{n-p},
\end{equation*}
and (\ref{Feq1}) becomes
\begin{equation}\label{pgtr1}
x\cdot\nabla w(x)=\pm \frac{n-p}{p} w(x),\quad
w(x)=u(x)-u(f(x)),\quad
\text{a.e.}\;\Gamma
\end{equation}
or equivalently
\begin{equation*}
{x}\cdot \nabla
\left(|x|^{\mp\frac{n-p}{p}}w(x)
\right)=0\quad
\text{a.e.}\;\Gamma.
\end{equation*}
Thus a maximizer of (\ref{mainineq1}), if exists,
has to satisfy
\begin{equation*}
{u}(x)={u}(f(x))+|x|^{\pm\frac{n-p}{p}}
\psi\left(x\right)
\end{equation*}
for some $\psi$ that solves
\begin{equation}\label{psradinv}
x\cdot\nabla\psi=0\;\;\text{\big(equivalently}\;\;
\frac{\partial}{\partial |x|}\psi=0\big)\quad \text{on}
\;\; \,\Gamma.
\end{equation}
By this radial invariance of $\psi$ in $\Omega$,
we can write $\psi(x)=\psi({x}/{|x|})$, after extending
it radially, wherever necessary, to the unit ball.
Moreover, with a sufficiently small $\eta>0$,
one can fit $B(\eta)$ inside $\Omega$.
So if ${u}(x)={u}(f(x))+|x|^{-\frac{n-p}{p}}
\psi\left(x\right)$ then
both sides of (\ref{mainineq1}) dominate
\begin{equation*}
\int_{B(\eta)}\frac{|\psi(x)|^{n}}{
|x|^{n}}\,dx=
\int_{\mathbb{S}^{n-1}}|\psi(\omega)|^{n}d\omega
\int_{0}^{\eta}\frac{1}{r}\,dr =
+\infty
\end{equation*}
unless $\psi=0$. Therefore, for $u$ to maximize
(\ref{mainineq1}), it must satisfy
\begin{equation}\label{umax}
{u}(x)={u}(f(x))+|x|^{\frac{n-p}{p}}
\psi(x)\quad \text{a.e.}\;\Omega,
\end{equation}
assuming $u-u\circ f\neq 0$ a.e. $\Omega$.\\
\indent Recall that (\ref{mainineq1})
lacks nontrivial maximizers when
$u-u\circ f\in W^{1,p}_{0}(\Omega)$.
To avoid that, take an
$\psi\in C\left(\bar{\Omega}_0\right)$
that satisfies (\ref{psradinv}).
If $|x|^{\frac{n-p}{p}}
\psi \in W^{1,p}(\Omega)$ then
$u-u\circ f\in W^{1,p}(\Omega)\cap
C\left(\bar{\Omega}_0\right)$. However
$ u(x)-u(f(x))\neq 0$ for every $
x\in \partial\Omega$, for otherwise
$\psi$ vanishes identically on
$\partial\Omega$ which, by (\ref{psradinv}),
implies $\psi$ vanishes on $\bar{\Omega}_{0}$.
So assume $\psi\in C\left(\bar{\Omega}_0\right)$
henceforth.\\
\indent  A  nontrivial maximizer of (\ref{mainineq1}) cannot be in
$ C_{\text{rad}}(\bar{\Omega}_0)$,
the space of continuous in the radial direction functions
on $\bar{\Omega}_{0}$. Indeed,
since $f\in C(\bar{\Omega}_{0})$
and $f(x)=x$ on $\partial\Omega$, then for a function
$u\in C_{\text{rad}}(\bar{\Omega}_0)$ that
verifies (\ref{umax}), we would have $u(x)-u(f(x))=0$
on $\partial \Omega$. Let
\begin{equation*}
\phi(x):=\left\{
  \begin{array}{ll}
\vspace{0.2 cm} |f(x)|^{\frac{n-p}{p}}\psi(x),
& \hbox{in $\Omega_{0}$;} \\
    0,& \hbox{on $\partial\Omega$.}
  \end{array}
\right.
\end{equation*}
Then $\phi\in C\left(\Omega_{0}\right)$, and is
evidently radially invariant in $\Omega_{0}$.
Since $f\neq 0$  in $\Omega_{0}$,
then $\phi\notin C_{\text{rad}}(\bar{\Omega}_0)$.
Define
\begin{equation}\label{fin1}
{\xi}(x):=|x|^{\frac{n-p}{p}}
\psi(x)+\phi(x),\quad x\in \bar{\Omega}_0.
\end{equation}
Notice that either $\xi\notin C_{\text{rad}}(\bar{\Omega}_0)$, or else
$\xi$ is identically zero. Since
\begin{equation*}
\psi(f(x))=\psi(x),\;\;\phi(f(x))=0
\end{equation*}
for every $x\in\Omega_{0}$, then
\begin{equation*}
{\xi}(f(x))=|f(x)|^{\frac{n-p}{p}}
\psi(x)=\phi(x)\quad \text{in}\;\,\Omega_{0}
\end{equation*}
and $\xi$ satisfies (\ref{umax}). (See Figure A).
\begin{center}
\begin{tikzpicture}[scale=0.5,
use Hobby shortcut,closed=true]
\draw [densely dotted,thick,black]
 (-3.5,0.5) .. (-3,2.5) .. 
(-1,3.5).. (1.5,3).. (4,3.5).. 
(5,2.5).. (5,0.5) ..(2.5,-2).. 
(0,-3.5).. (-3,-2).. (-3.5,0.5);
\draw[thick,black] (0.0, 0)-- (5,2.5) node[black,right] {$p$};
\draw (2,1) node[black,below] {$x$};
\draw (2,1) node[black,thick] {\tiny{-}};
\draw[thick,black] (0.0,0.1) node[black,below] {$O$};
\draw[thick,black] (0.0, -0.5) node[black,below]
 {\tiny{(The origin)}};
\draw[thick,black] (3,-2) node[black,right]
 {$\Omega$};
\draw (1.5,-4.4) node[black,thick]
 {Figure A: $f(x)=f(p)=p$, $\,\xi(f(x))=\xi(p)=
|p|^{\frac{n-p}{p}}\psi(p)$.};
 \end{tikzpicture}
\end{center}
Thus, (\ref{fin1}) provides a maximizer of
(\ref{mainineq1}) for all $1\leq p<n$
as long as $f$ and $\psi$ are such that
$\psi\in L^{\infty}(\Omega)$,
$|x|^{\frac{n-p}{p}}\psi$, $|f|^{\frac{n-p}{p}}\psi$
are in $W^{1,p}(\Omega)$. For these maximizers,
both sides of (\ref{mainineq1}) equal
\begin{equation*}
\int_{\Omega}\frac{{|\psi|^{p}}}{|x|^{2p-n}}
\,dx\leq \|\psi\|_{L^{\infty}(\Omega)}^{p}
\int_{\Omega}\frac{1}{|x|^{2p-n}}
\,dx<\infty
\end{equation*}
for all $1\leq p< n$. Interestingly,
since $\xi \notin C_{\text{rad}}(\bar{\Omega}_{0})$
then $\xi-\xi\circ f$ cannot be a $W^{1,1}_{0}(\Omega)$
function.
This is because
$\xi-\xi\circ f \in C(\Omega_{0})
\setminus C(\bar{\Omega}_{0})$ and,
although
$\xi(x)-\xi(f(x))=0$, for every $x\in \partial \Omega$,
its trace $|x|^{\frac{n-p}{p}}\psi \neq 0$
on $\partial \Omega$.
See Example \ref{example1}.\\
\indent
Iterating this approach,
maximizers of (\ref{mainineq2}), if exist,
verifiably take the form
\begin{equation}\label{umaxc}
{u}(x)={u}(f(x))+{\psi(x)}
{\left(\log{\frac{M}{|x|}}\right)^{-\frac{n-1}{n}}}\quad \text{a.e.}\;\Omega
\end{equation}
with an $\psi$ that satisfies (\ref{psradinv}),
and will be momentarily further determined.
Observe that we excluded the functions
${u}(x)={u}(f(x))+{\psi(x)}
{\left(\log{\frac{M}{|x|}}\right)^{\frac{n-1}{n}}}$
with which both sides in (\ref{mainineq2}) diverge.
Indeed, for $0<\eta<M$ small enough,
both sides of (\ref{mainineq2}) dominate
\begin{equation*}
\int_{B(\eta)}\frac{|\psi(x)|^{n}}{
|x|^{n}\log{\frac{M}{|x|}}}=
\int_{\mathbb{S}^{n-1}}|\psi(\omega)|^{n}
\int_{0}^{\eta}\frac{1}
{r\log{\frac{M}{r}}}\,dr d\omega=
+\infty.
\end{equation*}
With $u$ in (\ref{umaxc}), both sides of (\ref{mainineq2})
equal
\begin{eqnarray*}
I_{\Omega}:=\int_{\Omega}\frac{|\psi|^{n}}
{|x|^{n}
\left(\log{\frac{M}{|x|}}\right)^{2n-1}}\,dx.
\end{eqnarray*}
Whether $I_{\Omega}$ converges is a little tricky.
Let's find $\psi$ that makes $I_{\Omega}$ converge.
Such $\psi$ undoubtedly depends on $\partial\Omega$
where the logarithmic singularity lies.
The radial invariance of $\psi$, and the radial symmetry of its factor
suggest writing $I_{\Omega}$
in spherical coordinates.
So, let $\partial \Omega$
be given by $r=r(\omega)$.
Since $\Omega$ is $C^{1}$
then $r(\omega)\in
C^{1}(\mathbb{S}^{n-1})$.
And since $0\in \Omega$ then
there exists $r_{0,\Omega}>0$
such that $r(\omega)\geq r_{0,\Omega}$
on $\mathbb{S}^{n-1}$. Therefore
\begin{equation*}
\hspace{-0.5 cm}
I_{\Omega}=\int_{\mathbb{S}^{n-1}}
\int_{0}^{r(\omega)}\frac{|\psi(\omega)|^{n}}
{r{\left(\log{\frac{M}{r(\omega)}}\right)^{2n-1}}}\,dr d\omega=
\frac{1}{2n-2}
\int_{\mathbb{S}^{n-1}}\frac{|\psi(\omega)|^{n}}
{\left(\log{\frac{M}{r(\omega)}}\right)^{2n-2}}\,d\omega.
\end{equation*}
Let $\psi \in C(\bar{\Omega}_{0})$.
If $\Lambda:=\{x\in \partial\Omega:|x|=M\}$ has
positive Lebesgue surface measure, $S(\Lambda)$, and $\psi\neq 0$ on $\Lambda$,
then $I_{\Omega}$ is ill-defined. Suppose that
$\Lambda$ is, in addition, simply connected. Then,
for $I_{\Omega}$ to be well-defined,
$\psi$ must vanish in the cone with apex at $0$
and base on $\Lambda$.
Obviously $\Lambda=\partial\Omega$ iff
$\Omega=B(M)$. Thus, the inequality (\ref{mainineq2})
on balls admits no maximizers. \\
\indent Assume that $S(\Lambda)=0$.
We need $|\psi|^n$ to decrease
at least as fast as the now radially invariant
logarithmic factor
$\left(\log{\left({M}/{r(\omega)}\right)}
\right)^{(2n-2)}$
when $r(\omega)\rightarrow M$.
Luckily, if $g(t):=(1-t)^{n\alpha}/\left(\log{(1/t)}
\right)^{2n-2}$, with $\alpha\geq (2n-2)/n$, $0<t<1$,
then $g\in C\left( ]0,1[ \right)\cap
L^{\infty}\left([b,1]\right)$, for any fixed $b>0$.
Also if $h(x):=r(x)/M$, $x\in \mathbb{S}^{n-1}$,
then $h\in C\left(\mathbb{S}^{n-1}\right)$,
and consequently $g\circ h \in C(\mathbb{S}^{n-1})$.
This urges the choice
$\psi(x)= \left(M-r(x/|x|)\right)^{\alpha}$,
$x\in\Omega_{0}$, for which $I_{\Omega}$ converges.
Suppose
\begin{equation}\label{uc}
{\eta}(x):=
\left\{
  \begin{array}{ll}
\vspace{0.25 cm}{\left(M-r
\left(\dfrac{x}{|x|}\right)\right)^{\alpha}}
{\left(\log{\dfrac{M}{|x|}}\right)^{-\frac{n-1}{n}}},
& \hbox{in $\Omega_{0}$;} \\
    0, & \hbox{on $\partial\Omega$.}
  \end{array}\right.
\end{equation}
Then $\eta(f(x))=0$ for every $x\in\bar{\Omega}_{0}$
and we get
\begin{equation*}
\eta(x)-\eta(f(x))=\left(M-r
\left(\frac{x}{|x|}\right)\right)^{\alpha}
{\left(\log{\frac{M}{|x|}}\right)^{-\frac{n-1}{n}}},
\quad x\in \Omega_0.
\end{equation*}
Hence, if $\partial \Omega$ has the representation
$r=r(x/|x|)$ with $r(x/|x|)<M$, $S-$a.e
on $\partial\Omega$, then $\eta$ in
(\ref{uc}) is a maximizer of (\ref{mainineq2})
 for all $n\geq2$,
provided we choose $\alpha\geq (2n-2)/n$ large enough
that $\eta \in W^{1,n}(\Omega)$. Observe here that
$\eta-\eta\circ f\in C(\Omega_{0})\setminus
C_{\text{rad}}(\bar{\Omega}_{0})$, with its trace
$ \left(M-r \left({x}/{|x|}\right)\right)^{\alpha}
{\left(\log{{M}/{|x|}}\right)^{-\frac{n-1}{n}}}>0$,
$S-$a.e $\partial\Omega$. Thus
 $\eta-\eta\circ f\notin W^{1,1}_{0}(\Omega)$
where (\ref{mainineq2}) has no maximizers.
See Example \ref{example2}.
\section{Examples}
\begin{exm}\label{example1}
Let $n\geq 3$, and define on the ellipsoid
$E_{a}$ in Section \ref{apps} the functions
$\psi(x)=x_{1}/|x|_{a}$ and
\begin{equation*}
\phi(x):=\left\{
  \begin{array}{ll}
    |f_{a}(x)|^{\frac{n-p}{p}}\psi(x)=
\left({|x|}/{|x|_{a}}\right)^{\frac{n-p}{p}}\psi(x),
& \hbox{in $E_{a}\setminus\{0\}$;}\\
    0,& \hbox{on $\partial E_{a}$.}
  \end{array}
\right.
\end{equation*}
Since $\phi,\psi\in L^{\infty}(E_{a})$ and
are both weakly differentiable, and since
\begin{equation*}
\max\left\{\left|\nabla \frac{|x|}{|x|_{a}}\right|,
\left|\nabla \frac{x_{1}}{|x|_{a}}\right|,
\left|\nabla |x|^{\frac{n-p}{p}}\right|
\right\} \lesssim_{a}
\frac{1}{|x|},\quad x\neq 0,
\end{equation*}
then $
\xi(x)=|x|^{\frac{n-p}{p}}
\psi+\phi$ is in $W^{1,p}(E_{a})$ for all $1\leq p<n$.
We note that $\psi\in
C^{\infty}(\bar{E_{a}}\setminus\{0\})$ and
$\xi\notin C_{\text{rad}}(\bar{E_{a}}\setminus\{0\})$.
We also find
\begin{equation*}
\xi(f_{a}(x))=|f_{a}(x)|^{\frac{n-p}{p}}
\psi(x)=\phi(x), \quad \text{in}\;\,E_{a}\setminus\{0\}.
\end{equation*}
Consequently
\begin{equation*}
\xi(x)-\xi(f_{a}(x))=|x|^{\frac{n-p}{p}}
\psi(x), \quad \text{in}\;\,E_{a}\setminus\{0\}.
\end{equation*}
Moreover
\begin{eqnarray*}
\int_{E_{a}}
\frac{|\xi-\xi\circ f_{a}|^{p}}{|x|^{p}}\,dx
&=&\left(\frac{p}{n-p}\right)^{p}\int_{E_{a}}\left|\frac{x}{|x|}
\cdot \nabla \xi \right|^{p}
\,dx\\\\
&=&\int_{E_{a}}\frac{{|x|^{n-p}}{|\psi|^{p}}}{|x|^{p}}
\,dx\leq \frac{1}{a^{p}_{1}} \int_{E_{a}}\frac{dx}{|x|^{2p-n}}
<\infty
\end{eqnarray*}
for all $1\leq p<n$.
\end{exm}
\begin{exm}\label{example2}
Again, consider the ellipsoid $E_{a}$.
Using spherical coordinates,
$\partial E_{a}$ has the representation
\begin{equation*}
r=r(\omega)={\left|\omega_{a}\right|^{-1}},\quad
\omega_{a}:=\left(a_{1}\omega_{1},...,
a_{n}\omega_{n}\right).
\end{equation*}
Losing no generality, let $a_{1}=\min_{i} a_{i}$,
$a_{2}=\max_{i} a_{i}$ so that $M_{a}=1/a_{1}$ and
\begin{equation*}
1\leq \frac{M_{a}}{r(\omega)}
\leq M_{a} a_{2},\quad \text{for
every}\;\; \omega\in \mathbb{S}^{n-1}.
\end{equation*}
Assume that $a_{1}<a_{i}$ for every $2\leq i\leq n$.
Then $r(\omega)=M_{a}$ precisely at the two points
$p_{\pm}=(\pm {1}/{a_{1}},0,...,0)$.
Fix $\alpha\geq (2n-2)/n$, and let
\begin{equation*}
\psi(x)=\left\{
  \begin{array}{ll}
\vspace{0.2 cm}
\left(M_{a}-r\left(\dfrac{x}{|x|}
\right)\right)^{\alpha},
 & \hbox{in $E_{a}\setminus\{0\}$;} \\
    0, & \hbox{on $\partial E_{a}$.}
  \end{array}
\right.
\end{equation*}
Then $\psi \in C^{\infty}(\bar{E_{a}}
\setminus\{0\})\cap L^{\infty}(E_{a})$.
Computations show that
\begin{equation*}
{\eta}(x)={\psi(x)}
{\left(\log{\frac{M_{a}}{|x|}}\right)^{-\frac{n-1}{n}}}
\end{equation*}
satisfies
\begin{equation*}
\int_{E_{a}}\frac{|\eta-\eta\circ f_{a}|^{n}}
{|x|^{n}
\left(\log{\frac{M_{a}}{|x|}}\right)^{n}}\,dx=
\left(\frac{n}{n-1}\right)^{n}
\int_{E_{a}} \left|\frac{x}{|x|}
\cdot\nabla \eta\right|^{n} dx=I_{E_{a}}
\end{equation*}
where
\begin{equation*}
\hspace{-0.5 cm}
\label{finite1}
I_{E_{a}}=\int_{E_{a}}\frac{|\psi|^{n}}
{|x|^{n}\left(\log{\frac{M_{a}}{|x|}}\right)^{2n-1}}\,dx
=\frac{M^{n}_{a}}{2n-2}
\int_{\mathbb{S}^{n-1}}
\frac{\left(1-\frac{r(\omega)}
{M_{a}}\right)^{n\alpha}}{
\left(\log{\frac{M_{a}}{
r(\omega)}}\right)^{2n-2}}\,d\omega,
\end{equation*}
which converges since
$r(\omega)\geq {a^{-1}_{2}}>0$ uniformly,
and its integrand is the composite of
the continuous bounded functions
$\omega\mapsto r(\omega)/M_{a}=1/
\left(M_{a}|\omega_{a}|\right)$, and
\begin{equation*}
x\mapsto\left\{
     \begin{array}{ll}
       (1-x)^{n\alpha}/
       \left(\log{\left({1}/{x}\right)}\right)^{2n-2},
        & \hbox{$0<x< 1$;} \\
       0, & \hbox{$x= 1$.}
     \end{array}
   \right.
\end{equation*}
Evidently $\eta\in L^{\infty}(E_{a})$ when
$\alpha\geq (n-1)/n$. Let us check
$\nabla \eta\in L^{n}(E_{a})$ for all $n\geq2$.
For all $x\in E_{a}\setminus\{0\}$, we have
\begin{eqnarray*}
\left|\nabla \psi\right|&=&\alpha
\left(M_{a}-r\left(\frac{x}{|x|}\right)\right)^{\alpha-1}
\left|\nabla \frac{|x|}{|x|_{a}}\right|
\lesssim_{a,\alpha}
\frac{\psi^{\frac{\alpha-1}{\alpha}}}{|x|}.
\end{eqnarray*}
Thus
\begin{align*}
\hspace{-1cm}
\int_{E_{a}} |\nabla \eta|^{n} dx&
\lesssim_{a,\alpha,n}
\int_{E_{a}}\frac{\psi^{n}}{|x|^{n}
\left(\log{\frac{M_{a}}{|x|}}\right)^{2n-1}} dx +
\int_{E_{a}} \frac{\psi^{\frac{n(\alpha-1)}{\alpha}}}{|x|^{n}
\left(\log{\frac{M_{a}}{|x|}}\right)^{n-1}} dx\\
&\hspace{-1 cm}=\,\frac{M^{n}_{a}}{2n-2}
\int_{\mathbb{S}^{n-1}}
\frac{\left(1-\frac{r(\omega)}
{M_{a}}\right)^{n\alpha}}{
\left(\log{\frac{M_{a}}{
r(\omega)}}\right)^{2n-2}}\,d\omega+
\frac{M^{\frac{n(\alpha-1)}{\alpha}}_{a}}{n-2}
\int_{\mathbb{S}^{n-1}}
\frac{\left(1-\frac{r(\omega)}
{M_{a}}\right)^{\frac{n(\alpha-1)}{\alpha}}}{
\left(\log{\frac{M_{a}}{
r(\omega)}}\right)^{n-2}}\,d\omega<\infty
\end{align*}
for all $\alpha\geq n/2$. Noteworthily, $\eta
\in C(E_{a})\setminus C_{\text{rad}}(\bar{\Omega}_0)$,
and $\eta-\eta\circ f=\eta$ has positive trace. Hence
$\eta-\eta\circ f\notin W^{1,1}_{0}(E_{a})$.
\end{exm}
\section{Acknowledgement}
The author is grateful to Craig Cowan at the
university of Manitoba for his valuable comments
on the counterexample that proves
Proposition \ref{cntbtnd}.
\section*{References}

\end{document}